\documentclass{amsart}
\usepackage{amsmath}
\usepackage{mathrsfs}
\usepackage{leftidx}
\usepackage{amssymb}
\usepackage{latexsym}
\usepackage[all]{xy}
\usepackage{graphicx}
\usepackage{caption}

\newtheorem{theorem}{Theorem}[section]
\newtheorem{lemma}[theorem]{Lemma}
\newtheorem{proposition}[theorem]{Proposition}
\newtheorem{corollary}[theorem]{Corollary}
\theoremstyle{definition}
\newtheorem{definition}[theorem]{Definition}

\newtheorem{remark}[theorem]{Remark}
\numberwithin{equation}{section}

\begin{document}
\title{Admissible Coalgebras and Non-cosemisimple Hopf Algebras}
\author{Z.-P. Fan}
\address{the School of Mathematical Sciences, Ocean University of China,
Qingdao, 266100, P.R. China} \email{fanzp@ouc.edu.cn}

\subjclass[2000]{16T05, 16T15}



\keywords{the classification of finite dimensional Hopf algebras, coalgebra}

\begin{abstract}
In this paper, we give several necessary conditions for non-cosemisimple coalgebras being admissible. The implications simplify the classification problems for Hopf algebras of dimension 45, 105 and a few others.
\end{abstract}
\maketitle

\section*{Introduction}
This paper provides a new method for the classification of finite dimensional Hopf algebras, by showing necessary conditions of coalgebras being admissible. The ground field $\Bbbk$ is algebraically closed with $\mathrm{char}\ \Bbbk= 0$.

For its significant meaning in algebra and quantum physics, the classification of finite dimensional Hopf algebras has always been one of the central problems in the study of Hopf algebras. As the fundamental structure of Hopf algebras, coalgebras play an important role in the classification problem. Considering the property of the coradicals, we divide Hopf algebras into three kinds: cosemisimple, pointed and non-cosemisimple non-pointed. So far, cosemisimple Hopf algebras of dimensions $pq^2$ and $pqr$ and pointed Hopf algebras of dimension $pq^2$ have been classified by Natale and Andruskiewitsch\cite{natale1999semisimple,etingof2011weakly,andruskiewitsch2001counting}. Hilgemann and Ng prove that non-cosemisimple non-pointed Hopf algebras of dimension $2p^2$ are all duals of pointed Hopf algebras\cite{hilgemann2009hopf},  which completes the classification of such dimensions.

  The question of classifying Hopf algebras of a given dimension comes from the Kaplansky's ten conjectures (\cite{Kaplansky1975Bialgebras}). So far there are only a few general results.
The classification splits into several different parts according to the Hopf algebra being cosemisimple, pointed or non-pointed non-cosemisimple. A coalgebra is called `cosemisimple' if it equals its coradical, and it is called `pointed' if every simple subcoalgebra is $1$-dimensional. Hopf algebras which are both pointed and cosemisimple are group algebras. Then in the following discussions, by `pointed' we mean `pointed but not cosemisimple'.

For a finite dimensional Hopf algebra, by \cite{Larson1988Finite}, it is semisimle as an algebra if and only if it is cosemisimple as a coalgebra. There are a few important results on cosemisimple Hopf algebras, but there is no general strategy for their classification so far. For $p,q,r$ being odd primes, it is already known that cosemisimple Hopf algebras with dimension $p,2p,p^2,pq$ are trivial (\cite{zhu1994hopf,masuoka1995semisimple,masuoka1996pn,etingof1998semisimple}), and the classification of cosemisimple Hopf algebras with dimension $p^3,2p^2,pq^2,pqr$ are completed (\cite{masuoka1995semisimple68,masuoka1995self,natale1999semisimple,natale2001semisimple,natale2004semisimple,etingof2011weakly}).

At the moment, the most general method for the classification of pointed Hopf algebras is the lifting method developed by Andruskiewitsch and Schneider. First they decompose the associated graded Hopf algebra $gr(A)$ of a pointed Hopf algebra $A$ into a smash biproduct of a braided Hopf algebra $R$ and the group algebra $\Bbbk G(A)$, then study the structure of $R$ as a Nichols algebra over $\Bbbk G(A)$, finally lift relations in $R$ to get ones in $A$ (\cite{Andruskiewitsch2002Pointed,Andruskiewitsch1998Lifting,andruskiewitsch2000finite,Andruskiewitsch2010On}).

For $p,q$ being odd primes and $p<q\le 4p+11$, there is no non-cosemisimple non-pointed Hopf algebra with dimension $p,2p,p^2,pq$ (\cite{zhu1994hopf},\cite{ng2005hopf},\cite{Ng2001Non},\cite{ng2008hopf}).
For non-cosemisimple non-pointed Hopf algebras of other dimensions, there are several classification methods based on different subjects. The dimension of spaces related to the coradical (\cite{andruskiewitsch2001counting},\cite{beattie2004hopf},\cite{fukuda2008structure}), the order of the antipode (\cite{Ng2001Non},\cite{ng2004hopf},\cite{ng2005hopf},\cite{ng2008hopf},\cite{hilgemann2009hopf}), the Hopf subalgebras and the quotient Hopf algebras (\cite{natale2002hopf}) and the braided Hopf algebras in smash biproducts (\cite{cheng2011hopf}) all play important roles in the classification of non-cosemisimple non-pointed Hopf algebras.

The idea of block system is inspired by the classification technique used in \cite{andruskiewitsch2001counting}, \cite{beattie2004hopf} and \cite{fukuda2008structure}. For a Hopf algebra, we build the block system based on its coradical filtration. Block systems become interesting when the Hopf algebras are non-cosemisimple and non-pointed. In this paper, we focus on non-cosemisimple Hopf algebras with no nontrivial skew primitives, which form a subclass of non-cosemisimple non-pointed Hopf algebras. Corollary \ref{cor1}, \ref{cor2}, \ref{cor3}, \ref{cor4} and Proposition \ref{prop3} are the `rules' we know so far, which lead to our main results, Theorem \ref{thm1} and Theorem \ref{thm2}. Theorem \ref{thm1} shows a lower bound for the dimensions of non-cosemisimple Hopf algebras with no nontrivial skew primitives. By Radford's formula for $\mathcal{S}^4$, orders of group like elements are crucial in classification methods working with the antipode. So Theorem \ref{thm2} contributes to the classification of Hopf algebras of dimensions involved in it.

The block system shows a way to understand a non-cosemisimple Hopf algebra through its coalgebra structure by decomposing it into blocks which are highly related by the `rules'. We believe that by further study, more `rules' will be discovered, so more clearly we understand the Hopf algebra structure.

\section{The Building of a Block System}
For further discussions, we first construct for a non-cosemisimple coalgebra two abstract structures, which are called `tree structure' and `block system' respectively.

Let $(C,\Delta,\varepsilon)$ be a non-cosemisimple coalgebra over $\Bbbk$. According to \cite[Theorem 5.4.2]{montgomery1993hopf}, there exists a coideal $I$ and a coalgebra projection on the coradical $\pi: C\to C_0$, such that $C=C_0 \oplus I$ and $\ker \pi = I$. Fix this coideal $I$ for $C$, and set
\begin{equation}\label{eqrholrhor}
\rho_L = (\pi \otimes id)\Delta\ \text{和 }\ \rho_R = (id \otimes\pi)\Delta \ .
\end{equation}
See that $(C,\rho_L,\rho_R)$ becomes a $C_0$-bicomodule.

Let $\{C_n\}_{n\ge 0}$ denote the coradical filtration of $C$, and set
\begin{equation}\label{eqPn}
P_n= C_n \cap I \ .
\end{equation}
See more details for $P_n$ in \cite[Lemma 1.1]{andruskiewitsch2001counting}.

Let $\hat{C}$ denote the index set of simple subcoalgebras of $C$. For each $\tau \in \hat{C}$, there exists an integer $d_\tau$ such that $\dim D_\tau= {d_{\tau}}\!\!^2$. For convenience, let $D_g$ denote the simple subcoalgebra generated by the group like element $g\in G(C)$, so the index of $D_g$ in $\hat{C}$ is also $g$. For $d\ge 1$, set
\begin{equation}\label{eqC0d}
C_{0,d} = \bigoplus_{\tau\in \widehat{C},d_\tau = d} D_\tau \ .
\end{equation}

For each $\tau \in \hat{C}$, choose and fix a standard basis $\{e^\tau_{i,j}\}_{i,j\in\{1,\cdots,d_\tau\}}$ for simple subcoalgebra $D_\tau$, with the following structure:
\begin{equation}\label{eqstdbas}
\Delta(e^\tau_{i,j})=\sum_{k=1}^{d_\tau} e^\tau_{i,k}\otimes e^\tau_{k,j} \ ,\  \varepsilon(e^\tau_{i,j})=\delta_{i,j}.
\end{equation}
It is well known that $D_\tau$ has the unique simple right (resp. left) comodule isomorphism class $V_\tau$ (resp. $V_\tau^*$), and with the fixed standard basis $\{e^\tau_{i,j}\}_{i,j\in\{1,\cdots,d_\tau\}}$,
\begin{equation}\label{eqVtauVtaustar}
V_\tau \cong \Bbbk\{e^\tau_{1,j}\}_{j\in \{1,\cdots,d_\tau\}},\ V^*_\tau \cong \Bbbk\{e^\tau_{i,1}\}_{i\in \{1,\cdots,d_\tau\}} \ .
\end{equation}

Since $C_0$ is cosemisimple and $\Bbbk$ is a perfect filed, the $C_0$-bicomodule category ${^{C_0}}\!\mathcal{M}^{C_0}$ is semisimple. For $M\in {^{C_0}}\!\mathcal{M}^{C_0}$ and $\tau,\mu \in \widehat{C}$, consider $M$'s simple subbicomodules isomorphic to $V^*_\tau \otimes V_\mu$, and let $M^{\tau,\mu}$ denote the sum of such simple subbicomodules. Then
\begin{equation}\label{eqMtaumu}
M = \bigoplus_{\tau,\mu \in \widehat{C}} M^{\tau,\mu} \ .
\end{equation}

For $n\ge 1$, see that $(P_n,\rho_L,\rho_R)\in {^{C_0}}\!\mathcal{M}^{C_0}$.
Then
\begin{equation}\label{eqPtaumu}
P_n = \bigoplus_{\tau,\mu \in \widehat{C}} P^{\tau,\mu}_n \ .
\end{equation}
Since $P_{n-1}\subseteq P_n$ is a $C_0$-subbicomodule, then there exists another $C_0$-subbicomodule $Q_n\subseteq P_n$, such that
$$
P_n=P_{n-1}\oplus Q_n \ .
$$
Notice that as a $C_0$-bicomodule, $Q_n$ is unique up to isomorphisms; to be precise,
$$
Q_n\cong P_n/P_{n-1} \cong C_n/C_{n-1} \ .
$$

For $\tau, \mu \in \widehat{C}, n\ge 1$, there exists an index set $\Lambda_{\tau,
\mu, n}$, such that
\begin{equation}\label{eqxiaokuai}
Q^{\tau, \mu}_n = \bigoplus_{\lambda_{\tau, \mu, n}\in
\Lambda_{\tau, \mu, n}} V^{\tau, \mu}_{n, \lambda_{\tau, \mu, n}} \
,
\end{equation}
where $V^{\tau, \mu}_{n, \lambda_{\tau, \mu, n}}$ is the simple $C_0$-subbicomodule isomorphic to $V^*_\tau \otimes V_\mu$.

By the discussion above, we have the following equations:
\begin{equation}\label{eqcditsumdecom}
\begin{array}{rl}
C = & C_0\oplus I
\\
=& C_0 \oplus \bigcup_{n\ge 1} P_n
\\
=& C_0 \oplus \bigoplus_{n\ge 1} Q_n
\\
= & \bigoplus_{\sigma\in \widehat{C}} D_\sigma \ \oplus
\bigoplus_{n\ge 1} \Big( \bigoplus_{\tau, \mu \in \widehat{C},
}Q^{\tau, \mu}_n \Big)
\\
= &  \bigoplus_{\sigma\in \widehat{C}} D_\sigma \ \oplus
\bigoplus_{n\ge 1} \Big( \bigoplus_{\tau, \mu \in \widehat{C}} \big(
\bigoplus_{\lambda_{\tau, \mu, n}\in \Lambda_{\tau, \mu, n}}
V^{\tau, \mu}_{n, \lambda_{\tau, \mu, n}} \big)
\Big) \ .\\
\end{array}
\end{equation}

Call the direct sum in (\ref{eqcditsumdecom}) $C=\bigoplus_{\sigma\in
\widehat{C}} D_\sigma \ \oplus \bigoplus_{n\ge 1} \Big(
\bigoplus_{\tau, \mu \in \widehat{C},  }Q^{\tau, \mu}_n
\Big)$ \emph{\textbf{a tree structure}} of $C$.

For every $V^{\tau, \mu}_{n, \lambda_{\tau, \mu, n}}$,
by(\ref{eqVtauVtaustar}),  there exists one $C_0$-bicomodule isomorphism
\begin{equation}\label{eqvtauvmuisoetauemu}
f: \Bbbk \{e^\tau_{i, 1} \otimes e^\mu_{1, j}\}_{i\in \{1, \cdots,
d_\tau\}, j\in \{1, \cdots, d_\mu\}} \to V^{\tau, \mu}_{n,
\lambda_{\tau, \mu, n}} \ .
\end{equation}
Set $w_{i, j}=f(e^\tau_{i, 1} \otimes e^\mu_{1, j})$, then
\begin{equation}\label{eqrhorwij}
\begin{array}{rl}
\rho_R(w_{i, j}) =& (f\otimes id )(id \otimes
\mathit\Delta)f^{-1}(w_{i, j})
\\
=& (f\otimes id ) \Big( e^\tau_{i, 1} \otimes \big(
\sum_{l=1}^{d_\mu} e^\mu_{1, l} \otimes e^\mu_{l, j} \big) \Big)
\\
=& \sum_{l=1}^{d_\mu} w_{i, l}\otimes e^\mu_{l, j} \ ;
\end{array}
\end{equation}
and
\begin{equation}\label{eqrholwij}
\begin{array}{rl}
\rho_L(w_{i, j}) =& (id \otimes f)(\mathit\Delta\otimes id
)f^{-1}(w_{i, j})
\\
=& (id \otimes f) \Big( \big( \sum_{k=1}^{d_\tau} e^\tau_{i, k}
\otimes e^\tau_{k, 1} \big) \otimes e^\mu_{1, j} \Big)
\\
=& \sum_{k=1}^{d_\tau} e^\tau_{i, k}\otimes w_{k, j} \ .
\end{array}
\end{equation}
Set
\begin{equation}\label{eqye}
W^{\tau, \mu}_{n, \lambda_{\tau, \mu, n}}=\{w_{i, j}\}_{i\in \{1,
\cdots, d_\tau\}, j\in \{1, \cdots, d_\mu\}} \ ,
\end{equation}
then $W^{\tau, \mu}_{n, \lambda_{\tau, \mu, n}}$ is a basis of $V^{\tau, \mu}_{n,
\lambda_{\tau, \mu, n}}$.

For $n\ge 1$, set
\begin{equation}\label{eqgan}
W_n = \bigcup_{\tau, \mu \in \widehat{C}} W^{\tau, \mu}_n \ .
\end{equation}
then $W_n$ is a basis of $Q_n$.

For $n=0$, denote the union of the standard basis of $C$'s simple subcoalgebras by
\begin{equation}\label{eqW0}
W_0=\bigcup_{\sigma \in \widehat{C}, i_\sigma, j_\sigma\in\{1\,
\cdots, d_\sigma\}} \{e^\sigma_{i_\sigma, j_\sigma}\} \ ,
\end{equation}
then $W_0$ is a basis of $C_0$.

Set
\begin{equation}\label{eqW}
W = W_0\cup \bigcup_{n\ge 1} W_n  \ .
\end{equation}
by(\ref{eqcditsumdecom}),  $W$ is a basis of $C$, which is related to the tree structure of $C$.

Call $W$ in (\ref{eqW}) \textbf{\emph{a tree structure basis}} of $C$ derived by $W_0$ and $I$.

Set a map $\overline{\mathit\Delta}:C\to I\otimes I$ as followings:
$$
\overline{\mathit\Delta}(c) = \left\{
\begin{array}{ll}
(\mathit\Delta - \rho_L-\rho_R)(c), & \text{if }c\in I;
\\
0,  &\text{if }c\in C_0.
\end{array}
\right.
$$

For $w\in W$, if $w\in W\cap P_n$, $n\ge 1$, then
\begin{equation}\label{eqolDeltaw}
\overline{\mathit\Delta}(w)=\sum_{x, y\in W\cap P_{n-1}} k_{w, x, y}
x\otimes y,
\end{equation}
with only finite nonzero $k_{w, x, y}\in \Bbbk$. Set
\begin{equation}\label{eqoverlineLwRw}
\begin{array}{l}
\overline{\mathcal{L}}(w)=\left\{
\begin{array}{ll}
\{x\in W\cap I\ |\ k_{w, x, y}\ne 0, y\in W\cap I\}, & w\in W\cap I, \\
\varnothing,  & w\in W\cap C_0,
\end{array}\right.
\\
\overline{\mathcal{R}}(w)=\left\{
\begin{array}{ll}
\{y\in W\cap I \ |\ k_{w, x, y}\ne 0, x\in W\cap I\}, & w\in W \cap I, \\
\varnothing,  & w\in W\cap C_0.
\end{array}\right.
\end{array}
\end{equation}
These two sets will be used in the following discussions.

\begin{definition}\label{defbs}
For $n\ge 0$,
$$
B^{d_1, d_2}_n = \left\{
\begin{array}{ll}
0, &n=0,  d_1\ne d_2, \\
C_{0, d}, &n=0,  d_1=d_2=d\ge 1, \\
\bigoplus_{\tau, \mu \in \hat{C},  d_\tau=d_1, d_\mu=d_2}Q^{\tau,
\mu}_n,  &n, d_1 ,  d_2\ge1,
\end{array}\right.
$$
is called \textbf{\emph{a block}} of $C$ derived by $I$ and $Q_n$.

The sum $C= \bigoplus_{n\ge0, d_1, d_2\ge1} B^{d_1, d_2}_n$ is called \textbf{\emph{a block system}} of $C$ derived by $I$ and $\{Q_n\}_{n\ge1}$.
\end{definition}

\begin{proposition}\label{propnzdecom}
If there exixt $n>1$ and $d_1, d_2\ge1$, such that $B^{d_1, d_2}_n \ne 0$, then there exists a set of positive integers $\{b_1, \cdots,
b_{n-1}\}$, such that for $i = 1, \cdots,  n-1$, we have $B^{d_1, b_i}_i \ne
0$ and $B^{b_i, d_2}_{n-i} \ne 0$.
\end{proposition}
\begin{proof}
For any $\tau, \mu \in \widehat{C}, n\ge1$, every $Q^{\tau, \mu}_n$ are corresponding to a nonzero block $B^{d_1, d_2}_n$, where $d_1
=d_\tau,  d_2 =d_\mu$. By \cite[Lemma 3.2]{fukuda2008structure}, the proposition is proved.
\end{proof}

\begin{remark}\label{rmkdimblockinvariant}
As $C_0$-bicomodules, $Q^{\tau,\mu}_n\cong (C_n/C_{n-1})^{\tau,\mu}$. Although the building of a block system of $C$ depends on the choice of $I$ and $Q_n$, the dimension of a block with specified superscript and subscript is invariant for $C$. That is, if $B'^{d_1,d_2}_n$ is a block induced by $I'$ and $Q'_n$, then $\dim B^{d_1,d_2}_n = \dim B'^{d_1,d_2}_n$. So $\{\dim B^{d_1,d_2}_n\}_{n\ge0,d_1,d_2\ge1}$ should be considered as classification invariants for coalgebras.
\end{remark}

\section{Block system for Hopf algebras}
In this section, the idea of block system is applied for Hopf algebras. We discuss relations between blocks to study the structure of Hopf algebras, then give necessary conditions for non-cosemisimple coalgebras being admissible.

Assume that $H$ is a non-cosemisimple Hopf algebra of finite dimension. Follow notations in section 1.

For $\tau, \mu \in \widehat{H},  g\in G(H)$, denote subcoalgebras $\mathcal{S}(D_\tau)$,
$gD_\tau$ and $D_\tau g$ by $\mathcal{S}\tau$, $g\tau$ and $\tau g$ respectively.

Obviously, $gQ^{\tau, \mu}_n = Q^{g\tau,g\mu}_n$ and $\mathcal{S}(Q^{\tau, \mu}_n)=Q^{\mathcal{S}\mu,\mathcal{S}\tau}_n$. According to Definition \ref{defbs}, we have
\begin{equation}\label{eqgBd1d2n}
gB^{d_1, d_2}_n = B^{d_1, d_2}_n g= B^{d_1, d_2}_n,  \ \text{and}\
\end{equation}
\begin{equation}\label{eqSBd1d2n}
\mathcal{S}(B^{d_1, d_2}_n ) = B^{d_2, d_1}_n.
\end{equation}

\begin{proposition}\label{propBd1d2ndvdbyGH}
For any $n\ge 0$ and $d_1, d_2\ge1$, $\dim B^{d_1, d_2}_n$ is divided by $\dim B^{1, 1}_0=|G(H)|$.
\end{proposition}

\begin{proof}
For $n=0$, by the Nichols-Zoeller Theorem, $B^{d_1, d_2}_0$ is a free right $\Bbbk G(H)$-module, then $\dim B^{d_1,d_2}_0$ is divided by $|G(H)|$.

For $n\ge1$时, obviously $(H_{n-1}, \mathit\Delta,
m)$ is a right $(H, \Bbbk G(H))$-Hopf module, then $B^{d_1, d_2}_n\oplus
H_{n-1}$ is also a right $(H, \Bbbk G(H))$-Hopf module. Similarly by the Nichols-Zoeller Theorem, $H_{n-1}$ and $B^{d_1, d_2}_n\oplus H_{n-1}$ are free right $\Bbbk G(H))$-modules. Then $B^{d_1, d_2}_n$ is a free right $\Bbbk G(H)$-module. So $\dim
B^{d_1, d_2}_n$ is divided by $|G(H)|$.
\end{proof}

\begin{proposition}\label{propBd1nB1dndvdbydGH}
For any $n\ge1$, $d\ge 1$, $\dim
B^{d, 1}_n$ and $\dim B^{1, d}_n$ are divided by $d|G(H)|$.
\end{proposition}

\begin{proof}
According to Definition \ref{defbs}, $B^{d, 1}_n = \bigoplus_{\tau\in \hat{H},  d_\tau=d,  g\in G(H)}Q^{\tau, g}_n$. For any $h\in G(H)$, $hQ^{\tau, g}_n=Q^{h\tau, hg}_n\subset B^{d,1}_n$. By (\ref{eqxiaokuai}) and (\ref{eqye}),  $\dim Q^{\tau,
g}_n$ is a multiple of $d$. So $\dim B^{1, d}_n$ is divided by $d|G(H)|$.
\end{proof}

\begin{proposition}\label{propBd2d1ne0}
If there exists $n\ge 1$ and $d_1,
d_2\ge1$, such that $B^{d_1, d_2}_n \ne 0$, then $B^{d_2, d_1}_n \ne 0$.
\end{proposition}

\begin{proof}
By (\ref{eqSBd1d2n}), $\dim B^{d_1,
d_2}_n =\dim B^{d_2, d_1}_n $,  which means that $B^{d_1, d_2}_n \ne 0$ is equivalent to $B^{d_2, d_1}_n \ne 0$.
\end{proof}

\begin{proposition}\label{propBd1d3n2ne0}
Suppose that there exist $n_1 \ge1$ and $d_1,
d_2\ge1$, such that $B^{d_1, d_2}_{n_1} \ne 0$. If $d_1 \ne d_2$, then there exists $d_3\ge 1$ and $n_2\ge n_1+1$, such that $B^{d_1, d_3}_{n_2}\ne 0$.
\end{proposition}

\begin{proof}
It is well known that $B^{d_1, d_2}_{n_1} \ne
0$ if and only if there exist $\tau, \mu \in \widehat{H}$, such that $d_\tau=d_1, d_\mu
=d_2,  Q^{\tau, \mu}_{n_1}\ne 0$. If $d_\tau \ne d_\mu$, then by \cite[Lemma 3.8]{fukuda2008structure}, there exist a simple subcoalgebra $E$ of $H$ and $n_2 \ge n_1+1$, such that $Q^{\tau, E}_{n_2} \ne 0$. Set $d_3=\dim E$, then $B^{d_1,
d_3}_{n_2}\ne 0$.
\end{proof}

Recall $\overline{\mathcal{L}}(v)$ and $\overline{\mathcal{R}}(v)$ defined in \ref{eqoverlineLwRw}. Denote the dual Hopf algebra of $H$ by $H^*$, and the dual basis of $W$ by $\{w^*\ |\ w\in W\}$, respectively.

\begin{lemma}\label{lmaintegral}
For any $u\in W\cap I$,
\begin{enumerate}
\item if $\rho_L(u)= 1\otimes u$, and $u\notin \overline{\mathcal{R}}(v)$ for any $v\in W$, then $u^*$ is a left integral of $H^*$ and there exists $g\in G(H)$ such that $\rho_R(u)=u\otimes g$;
\item if $\rho_R(u)= u\otimes 1$, and $u\notin \overline{\mathcal{L}}(v)$ for any $v\in W$, then $u^*$ is a right integral of $H^*$ and there exists $h\in G(H)$ such that $\rho_L(u)=h\otimes u$.
\end{enumerate}
\end{lemma}
\begin{proof}
The proof of (2) is similar with the one of (1), so we just prove (1).
For any $w\in W$, consider the convolution product $w^**u^*$. For any $v\in W$,
since $u\notin \overline{R}(v)$, we have
$$
w^**u^*(v)=\left\{
\begin{array}{ll}
1& w=1, v=u, \\
0& \text{otherwise}.
\end{array}\right.
$$
Then $w^**u^*=w^*(1)u^*$. By the arbitrariness of $w$, $w^*$ may go through all the basis elements in $W^*$. That means for any $f\in H^*$, we have $f*u^*=f(1)u^*$. Then $u^*$is a left integral of $H^*$. Since the left integral space of $H^*$ is one dimensional, by (\ref{eqrhorwij}), there exists a group-like elements $g$, such that $\rho_R(u)=u\otimes
g$.
\end{proof}

Since $B^{1, 1}_0=\Bbbk G(H)$, then there exists $n\ge 0$, such that $B^{1, 1}_n\ne 0$. Set
\begin{equation}\label{B11mmax}
m^H=\max \{m\ |\ B^{1, 1}_m\ne 0\} \ .
\end{equation}

\begin{proposition}\label{propdimB11mH}
$\dim B^{1, 1}_{m^H} = |G(H)|$.
\end{proposition}

\begin{proof}
By \cite[Proposition 3.2 (i)]{beattie2013techniques}, $m^H\ge
1$. For any $u\in W\cap B^{1, 1}_{m^H}$ with $\rho_L(u)= 1\otimes u$, there exists $h\in G(H)$, $\rho_R(u)= u\otimes h$.

We prove by contradiction that for any $w\in W$, $u\notin
\overline{\mathcal{R}}(w)$. Suppose that $u\in\overline{R}(w_0)$ for some nonzero $w_0\in W_H$. Then $\rho_R(w_0)= w_0\otimes h$. Then there exists $r_1>m^H$, such that $w_0\in B^{d, 1}_{r_1}$. By Proposition \ref{propBd1d3n2ne0}, there exists $r_2\ge r_1>m^H$, such that $B^{1, 1}_{r_2}\ne 0$. It is contradicted with the definition of $m^H$.

By Lemma \ref{lmaintegral}, $u^*$ is a left integral of $H^*$. Since the left integral space of $H^*$ is one-dimensional, then $|W^{1,
h}_{m^H}|=1$ and for any grouplike element $g\ne h$, we have $|W^{1, g}_{m^H}|=0$ (that means $\dim Q^{1, h}_{m^H}=1$ and $Q^{1, g}_{m^H}=0$). By acting $G(H)$ on $Q^{1,
h}_{m^H}$, we get $B^{1, 1}_{m^H}$. So $\dim B^{1, 1}_{m^H} =|G(H)|$.
\end{proof}

By the propositions above, we have the necessary conditions for non-cosemisimple coalgebras being admissible:

\begin{theorem}\label{thmnc1}
If a finite dimensional non-cosemisimple coalgebra $C$ is admissible, then every block system of $C$ must satisfy Proposition \ref{propBd1d2ndvdbyGH},Proposition \ref{propBd1nB1dndvdbydGH},Proposition \ref{propBd2d1ne0},Proposition \ref{propBd1d3n2ne0} and Proposition \ref{propdimB11mH}.
\end{theorem}

\begin{proof}
See that $C$ being admissible means that there is a Hopf algebra structure compatible with the coalgebra structure of $C$. Then every block system of $C$ is actually a block system of a Hopf algebra, which must satisfy Proposition \ref{propBd1d2ndvdbyGH},Proposition \ref{propBd1nB1dndvdbydGH},Proposition \ref{propBd2d1ne0},Proposition \ref{propBd1d3n2ne0} and Proposition \ref{propdimB11mH}.
\end{proof}

See that $P^{\tau,\mu}_n$ is non-degenerate if and only if $Q^{\tau,\mu}_n \ne 0$ for $\tau,\mu \in \hat {H}$ and $n\ge 1$. Rewrite some results in \cite{beattie2013techniques} and \cite{fukuda2008structure}.

\begin{proposition}\label{prop1}
For $\tau,\mu \in \hat {H}$:
\begin{enumerate}
\item \cite[Lemma 3.2]{fukuda2008structure} For some $n>1$, if $Q^{\tau,\mu}_n \ne 0$, then there exists a set of simple subcoalgebras $\{D_1, \cdots,D_{n-1}\}$ such that $Q^{\tau,D_i}_i \ne 0$ and $Q^{D_i,\mu}_i \ne 0$ for all $1\le i\le n-1$.
\item \cite[Lemma 3.5]{fukuda2008structure} Let $\mathcal{S}\tau$ denote the simple subcoalgebra $\mathcal{S}(D_\tau)$ and $g\tau$ (resp. $\tau g$) denote the simple subcoalgebra $gD_\tau$ (resp. $D_\tau g$) for $g\in G(H)$.  Then $\dim Q^{\tau,\mu}_n = \dim Q^{\mathcal{S}\tau,\mathcal{S}\mu}_n =\dim Q^{g\tau,g\mu}_n =\dim Q^{\tau g,\mu g}_n $, for $n\ge1$.
\item \cite[Lemma 3.8]{fukuda2008structure} Assume that $Q^{\tau,\mu}_{n_1}\ne 0$ for some $n_1\ge1$. If $d_\tau \ne d_\mu$ or $\dim Q^{\tau,\mu}_{n_1}\ne d_\tau^2$, then there exists a simple subcoalgebra $E$ of $H$ such that $Q^{\tau,E}_{n_2} \ne 0$ for some $n_2 \ge n_1+1$.
\end{enumerate}
\end{proposition}

\begin{proposition}\cite[Proposition 3.2 (i)]{beattie2013techniques}\label{prop2}
If $H$ is a non-cosemisimple Hopf algebra with no nontrivial skew-primitives, then for any $g\in G(H)$,
\begin{enumerate}
\item there exists $h\in G(H)$, such that $Q^{g,h}_m\ne 0$ for some $m>1$;
\item there exist $\tau,\mu \in \hat {H}$ with $d_\tau=d_\mu > 1$, such that $Q^{\tau,g}_1\ne 0$, $Q^{\tau,\mu}_k\ne0$ for some $k>1$.
\end{enumerate}
\end{proposition}

\begin{remark}\label{rmk1}
A pointed Hopf subalgebras is called nontrivial if it is not a group algebra. By \cite[Proposition 1.8]{andruskiewitsch2001counting}, $H$ has no nontrivial skew-primitives if and only if it has no nontrivial pointed Hopf subalgebras. Following \cite[Lemma 2.8]{beattie2013classifying}, if $(|G(H)|, \dim H / |G(H)|)=1$, then $H$ has no nontrivial skew-primitives. An example is that $\dim H$ is free of squares.
\end{remark}

The following results show properties of the blocks of a Hopf algebra, which should be considered as `rules' for block systems.

\begin{corollary}\label{cor1}
For blocks of $H$ and $g\in G(H)$, we have
$$
gB^{d_1,d_2}_n = B^{d_1,d_2}_n g= B^{d_1,d_2}_n, \ \text{and}\
\mathcal{S}(B^{d_1,d_2}_n ) = B^{d_2,d_1}_n,
$$
with $n\ge0$ and $d_1,d_2\ge1$.

\end{corollary}
\begin{proof}
The image of a simple subcoalgebra of $H$ under $\mathcal{S}$ or left (right) action of $g\in G(H)$ is still a simple one of the same dimension. Hence the proposition is correct for $n=0$.

$\mathcal{S}$, $\mathcal{S}^{-1}$ and left action of $g\in G(H)$ are bijections and keep the coradical filtration. For $n\ge1$, following the definition of $Q^{\tau,\mu}_n$, we have $gQ^{\tau,\mu}_n = Q^{g\tau,g\mu}_n$ and $\mathcal{S}(Q^{\tau,\mu}_n)=Q^{\mathcal{S}\mu,\mathcal{S}\tau}_n$. By Proposition \ref{prop1}(2), $gQ^{\tau,\mu}_n \subseteq B^{d_1,d_2}_n$ and $\mathcal{S}(Q^{\tau,\mu}_n)\subseteq B^{d_2,d_1}_n$  with $d_\tau=d_1, d_\mu=d_2$, which leads to $gB^{d_1,d_2}_n\subseteq B^{d_1,d_2}_n$ and $\mathcal{S}(B^{d_1,d_2}_n ) \subseteq B^{d_2,d_1}_n$. Similarly $g^{-1}B^{d_1,d_2}_n\subseteq B^{d_1,d_2}$ and $\mathcal{S}^{-1}(B^{d_1,d_2}_n ) \subseteq B^{d_2,d_1}_n$.
Hence $gB^{d_1,d_2}_n = B^{d_1,d_2}_n$ and $\mathcal{S}(B^{d_1,d_2}_n ) = B^{d_2,d_1}_n$.

Similar for $B^{d_1,d_2}_n g= B^{d_1,d_2}_n$.
\end{proof}

\begin{corollary}\label{cor2}
For blocks of $H$, $\dim B^{1,1}_0=|G(H)|$ divides the dimension of every block of $H$.
\end{corollary}
\begin{proof}
Let $B^{d_1,d_2}_n$ be a block of $H$. By Corollary \ref{cor1}, every block of $H$ is stable under the left multiplication of $g\in G(H)$.

For $n=0$, $B^{d_1,d_2}_0$ with $d_1=d_2$ is a left $(H,\Bbbk G(H))$-Hopf module with comultiplication of $H$ as comodule map. Then by the Nichols-Zoeller Theorem, $B^{d_1,d_2}_0$ is a free left $\Bbbk G(H)$-module. Then $|G(H)|$ divides the dimension of $B^{d_1,d_2}_0$.

For $n>0$, we have
\begin{align*}
\Delta(B^{d_1,d_2}_n)&\subset H_0\otimes (B^{d_1,d_2}_n)+  (B^{d_1,d_2}_n)\otimes H_0 +\sum_{1\le i \le n-1} P_i \otimes P_{n-i}\\
&\subset H_0\otimes (B^{d_1,d_2}_n)+  (B^{d_1,d_2}_n)\otimes H_0 + H_{n-1}\otimes H_{n-1}.
\end{align*}
See that $H_{n-1}$ and $B^{d_1,d_2}_n\oplus H_{n-1}$ are left $(H,\Bbbk G(H))$-Hopf modules with left multiplication of $\Bbbk G(H)$ as module map and comultiplication of $H$ as comodule map. Then we have $H_{n-1}$ and $B^{d_1,d_2}_n\oplus H_{n-1}$ are free over $\Bbbk G(H)$ and $|G(H)|$ divides the dimensions of $H_{n-1}$ and $B^{d_1,d_2}_n\oplus H_{n-1}$, which leads to the fact that $B^{d_1,d_2}_n$ is free over $\Bbbk G(H)$ and $|G(H)|$ divides the dimension of $B^{d_1,d_2}_n$.
\end{proof}

Every nonzero $Q^{\tau,\mu}_n$ leads to a nonzero block $B^{d_1,d_2}_n$ with $d_1 =d_\tau, d_2 =d_\mu$. Hence the following corollary is a direct result from Proposition \ref{prop1}.

\begin{corollary}\label{cor3}
For blocks of $H$, we have the following results.
\begin{enumerate}
\item If $B^{d_1,d_2}_n \ne 0$ for some $n>1$, then there exists a set of positive integers $\{b_1, \cdots,b_{n-1}\}$ such that $B^{d_1,b_i}_i \ne 0$ and $B^{b_i,d_2}_{n-i} \ne 0$ for all $i = 1,\cdots, n-1$.
\item If $B^{d_1,d_2}_n \ne 0$ for some $n\ge 1$, then $B^{d_2,d_1}_n \ne 0$.
\item Assume that $B^{d_1,d_2}_{n_1} \ne 0$ for some $n_1 \ge1$. If $d_1 \ne d_2$, then there exists positive integers $d_3$ such that $B^{d_1,d_3}_{n_2}\ne 0$ for some $n_2\ge n_1+1$.
\end{enumerate}
\end{corollary}

\begin{remark}\label{rmk2}
Since we always assume that $H$ is of finite dimension, then by Corollary \ref{cor3}(2),(3), $B^{d_1,d_2}_n\ne 0$ for $n\ge1$ with $d_1\ne d_2$ leads to  $B^{d_1,d_1}_{n_1}\ne 0$ and $B^{d_2,d_2}_{n_2}\ne 0$ for some $n_1,n_2\ge n+1$.
\end{remark}

Corollary \ref{cor4} is a direct result from Proposition \ref{prop2} and Corollary \ref{cor3} (2). It shows the necessary blocks for the block system of a non-cosemisimple Hopf algebra with no nontrivial skew-primitives.

\begin{corollary}\label{cor4}
If $H$ is a non-cosemisimple Hopf algebra with no nontrivial skew-primitives, then there exist $\tau,\mu \in \hat {H}$ with $d_\tau=d_\mu>1$, such that $B^{d_\tau,1}_1$, $B^{1, d_\tau}_1$, $B^{d_\tau,d_\mu}_k$ and $B^{1,1}_m\ne 0$ for some $k,m>1$.
\end{corollary}

\begin{remark}\label{rmk1}
A pointed Hopf subalgebras is nontrivial if it is not a group algebra. By \cite[Proposition 1.8]{andruskiewitsch2001counting}, a Hopf algebra has no nontrivial skew-primitives if and only if it has no nontrivial pointed Hopf subalgebras. Following \cite[Lemma 2.8]{beattie2013classifying}, if $(|G(H)|, \dim H / |G(H)|)=1$, then $H$ has no nontrivial skew-primitives. An example is that $\dim H$ is free of squares.
\end{remark}

For a pointed Hopf algebra, all the blocks are of the form $B^{1,1}_n$ with $n\ge0$. For a non-cosemisimple Hopf algebra with no nontrivial nontrivial skew-primitives, $B^{1,1}_n$'s also play an important role.

\begin{proposition}\label{prop3}
For blocks of $H$, let $m= \max \{m'\ |\ B^{1,1}_{m'}\ne 0\}$ and  $l= \min \{m'\ |\ B^{1,1}_{m'}\ne 0\}$, then
\begin{enumerate}
\item $\dim B^{1,1}_m = |G(H)|$;
\item if $l< m$, and $H$ is non-cosemisimple and has no nontrivial skew-primitives, then there exist $d_1,d_2,d_3,d_4>1$ and $l'>l>1$, such that $B^{d_1,1}_{l'}\ne 0$, $B^{1,d_2}_{l'}\ne 0$, $B^{d_1,d_3}_{l'-1}\ne0$ and $B^{d_4,d_2}_{l'-1}\ne0$.
\end{enumerate}
\end{proposition}
\begin{proof}
(1) If $H$ is cosemisimple, then $m=0$ and $B^{1,1}_0=\Bbbk G(H)$.

If $H$ is non-cosemisimple, then $m\ge 1$ by Corollary \ref{cor4}.

We built a basis for $H$ at first.

For any $\sigma \in \hat{H}$, let $\{e^\sigma_{i_\sigma,j_\sigma}\}_{i_\sigma,j_\sigma\in\{1\,\cdots,d_\sigma\}}$ be a standard basis of $D_\sigma$.

Recall that
$$
H=\bigoplus_{n'\ge0,d_1,d_2\ge1} B^{d_1,d_2}_{n'}=\bigoplus_{d\ge1}H_{0,d} \ \oplus \bigoplus_{n>0,d_1,d_2\ge1}\Big( \bigoplus_{\tau,\mu \in \hat{H}, d_\tau=d_1, d_\mu=d_2}Q^{\tau,\mu}_n \Big),
$$
where $Q^{\tau,\mu}_n = \bigoplus_k V^{\tau,\mu}_{n,k}$ with $V^{\tau,\mu}_{n,k} \simeq V_\tau\otimes V_\mu^*$ as simple $H_0$-bicomudules for some finite index $k\in \{1,\cdots,k(n,\tau,\mu)\}$. Here $k(n,\tau,\mu)$ denotes the number of simple $H_0$-bicomudules of the form $V_\tau\otimes V_\mu^*$ in $Q^{\tau,\mu}_n$.

We have that $D_\tau$ is the simple subcoalgebra of $H$ corresponding to $V_\tau$ with $\dim D_\tau = {d_\tau}\!\!^2$. Through the isomorphism form the unique simple left $D_\tau$-comodule to $V_\tau$, $\{e^\tau_{i_\tau,1}\}_{i_\tau\in\{1\,\cdots,d_\tau\}}$ leads to a basis of $V_\tau$. Similarly, $\{e^\mu_{1,j_\mu}\}_{j_\mu\in\{1\,\cdots,d_\mu\}}$ leads to a basis of $V_\mu^*$. Then $\{e^\tau_{i_\tau,1}\otimes e^\mu_{1,j_\mu}\}_{i_\tau\in\{1\,\cdots,d_\tau\},{j_\mu\in\{1\,\cdots,d_\mu\}}}$ leads to a basis of $V_\tau\otimes V_\mu^*$, and through the isomorphism from $V_\tau\otimes V_\mu^*$ to $V^{\tau,\mu}_{n,k}$, we have a basis of $V^{\tau,\mu}_{n,k}$, denoted by $W^{\tau,\mu}_{n,k}$.

Set
$$
W= \Big(\bigcup_{\sigma \in \hat{H},i_\sigma,j_\sigma\in\{1\,\cdots,d_\sigma\}} \{e^\sigma_{i_\sigma,j_\sigma}\}\Big) \ \cup \  \Big(\bigcup_{\tau,\mu\in \hat{H},n>0,k\in \{1,\cdots,k(n,\tau,\mu)\}} W^{\tau,\mu}_{n,k}\Big).
$$
Then $W$ is a basis of $H$.

Then we show that there is a basis element in $B^{1,1}_m$ such that its dual is the left integral of $H^*$.

Following notations at the beginning of this section, for any $w'\in W-H_0$, we have
$$
\Delta(w')= \rho_L(w')+\rho_R(w') + \sum_{x,y\in W} k_{x,y} x\otimes y,\ \text{for $k_{x,y}\in \Bbbk$}.
$$
Set
$$
\begin{array}{l}
L(w)=\left\{
\begin{array}{ll}
\{x\ |\ k_{x,y}\ne 0,y\in W\},& w\in W-H_0,\\
\varnothing, & w\in W\cap H_0,\ \text{and}\
\end{array}\right. \\
R(w)=\left\{
\begin{array}{ll}
\{y\ |\ k_{x,y}\ne 0,x\in W\},& w\in W-H_0,\\
\varnothing, & w\in W\cap H_0.
\end{array}\right.
\end{array}
$$

There exists $u\in W\cap B^{1,1}_m$, such that $\rho_L(u)+\rho_R(u)= 1\otimes u+ u\otimes g$ for some $g\in G(H)$.

We claim that $u\notin L(w)\cup R(w)$ for all $w\in W$. If $u\in R(w)$, for some $w\in W$, then by the proof of \cite[lemma 3.2]{fukuda2008structure}, $w\in B^{d,1}_r$, for $r>m$. By Corollary \ref{cor3} and Remark \ref{rmk2}, $B^{d,1}_r\ne0$ for $r>m$ leads to $B^{1,1}_{r'}\ne 0$ for $r'>m$, which is a contradiction with the maximality of $m$. Similar for $u\notin L(w)$.

Let $\{w^*\ |\ w\in W\}$ be the dual basis of $W$. Consider the convolution product $w^**u^*$. For any $v\in W$, by $u\notin R(v)$ and the definition of $W$, we have,
$$
w^**u^*(v)=\left\{
\begin{array}{ll}
1& w=1,v=u,\\
0& \text{otherwise}
\end{array}\right..
$$
Hence $w^**u^*=w^*(1)u^*$ for $w\in W$. Then for any $f\in H^*$, $f*u^*=f(1)u^*$, which makes $u^*$ a left integral of $H^*$.

Finally, we show that $\{hu\}_{h\in G(H)}$ is a basis of $B^{1,1}_m$.

See that $W\cap B^{1,1}_m$ is a basis of $B^{1,1}_m$. For $u'\in W\cap B^{1,1}_m$, assume that $\rho_L(u')+\rho_R(u')= g_1\otimes u'+ u'\otimes g_2$, $g_1,g_2\in G(H)$. Then there exists $k\in \Bbbk$ such that $kg_1^{-1}u' \in W$ and
$$
\rho_L(kg_1^{-1}u')+\rho_R(kg_1^{-1}u')= 1\otimes kg_1^{-1}u'+ kg_1^{-1}u'\otimes g_1^{-1}g_2.
$$
Since the space of left integral of $H^*$ is one-dimensional, we must have $kg_1^{-1}u'\in \Bbbk u$ and $g_1^{-1}g_2=g$. Then $u'=k'g_1u$ for some $k'\in\Bbbk$. Since $\rho_L+ \rho_R(\{hu\}_{h\in G(H)})$ is a linearly independent set in $H\otimes H$, $\{hu\}_{h\in G(H)}$ is also linearly independent. Then $\{hu\}_{h\in G(H)}$ is a basis of $B^{1,1}_m$ and $\dim B^{1,1}_m = |G(H)|$.

(2) Following notations in (1). Since $H$ has no nontrivial skew-primitives, we have $B^{1,1}_1 =0$ and $l>1$. By Corollary \ref{cor3} (1), the existence of $d_3,d_4$ with $B^{d_1,d_3}_{l'-1}\ne0$ and $B^{d_4,d_2}_{l'-1}\ne0$ is ensured by the existence of $d_1,d_2$.

For any $z\in W\cap B^{1,1}_l$, there exists $w\in W$ such that $z\in R(w)$ since $l<m$. Assume that $w\in B^{d,1}_t$.

If $d>1$, set $d_1=d_2=d$ and $l'=t$. By Corollary \ref{cor3} (2), $B^{1,d_2}_{l'}\ne0$. Then such $d_1,d_2$ satisfy the requirements.

Otherwise $d=1$. By the proof of \cite[lemma 3.2]{fukuda2008structure} and the minimality of $l$, we have $t\ge 2l$. For $B^{1,1}_t\ne0$, there exist $B^{1,d'}_{t-1},B^{d',1}_1 \ne 0$. Since $H$ has no nontrivial skew-primitives, we have $d'>1$. With $t-1\ge 2l-1 >l$, $B^{1,d'}_{t-1}$ and $B^{d',1}_{t-1}$ are not zero. Set $d_1=d_2=d'$ and $l'=t-1$. Then (2) is proved.
\end{proof}

\section{Applications}

By Corollary \ref{cor3} and Corollary \ref{cor4}, if $H$ is a non-cosemisimple Hopf algebra with no nontrivial skew-primitives, then its block system has the form as
$$
B^{1,1}_0 \oplus B^{d,d}_0\oplus B^{d,1}_1 \oplus B^{1,d}_1\oplus \cdots\oplus B^{1,1}_m \oplus B^{d,d}_{k},
$$
for some $d> 1$, $k= \max \{k'\ |\ B^{d,d}_{k'}\ne 0\}$ and $m= \max \{m'\ |\ B^{1,1}_{m'}\ne 0\}$. By Remark \ref{rmk2}, such a $k$ exits and if $B^{d,d'}_{n}\ne 0$ for some $d'\ge 1$, then $n\le k$. We use a diagram (Figure 1) to show the structure of the block system.

\unitlength=1cm
$$\begin{picture}(8.5,4.5)(0.1,0)
\put(0,0){\framebox(2.5,0.7)[]{$\mathbf{B^{1,1}_0}$}}
\put(3,0){\framebox(2.5,0.7)[]{$\mathbf{B^{d,d}_0}$}}
\put(3,1){\framebox(2.5,0.7)[]{$\mathbf{B^{d,1}_1}$}}
\put(6,1){\framebox(2.5,0.7)[]{$\mathbf{B^{1,d}_1}$}}
\put(3,3){\framebox(2.5,0.7)[]{$\mathbf{B^{1,1}_m}$}}
\put(6,3.5){\framebox(2.5,0.7)[]{$\mathbf{B^{d,d}_k}$}}
\put(4.3,2){\circle*{0.1}}
\put(4.3,2.3){\circle*{0.1}}
\put(4.3,2.6){\circle*{0.1}}
\put(7.3,2.3){\circle*{0.1}}
\put(7.3,2.6){\circle*{0.1}}
\put(7.3,2.9){\circle*{0.1}}
\put(6.8,0.4){\circle*{0.1}}
\put(7.3,0.4){\circle*{0.1}}
\put(7.8,0.4){\circle*{0.1}}
\end{picture}$$
\begin{center}
(Figure 1)
\end{center}
\vskip5mm

These six blocks are necessary for every non-cosemisimple Hopf algeabra with no nontrivial skew-primitives. If the block system of a coalgebra does not contain these six blocks above, then this coalgebra does not admits the structure of a non-cosemisimple Hopf algeabra with no nontrivial skew-primitives.

We find a lower bound for the dimension of a non-cosemisimple Hopf algebra with no nontrivial skew-primitives, which is a generation of \cite[Proposition 3.2]{beattie2013techniques}. For $n_1,n_2\in \mathbb{N}$, we denote by $lcm(n_1,n_2)$ the least common multiple of $n_1$ and $n_2$.

\begin{theorem}\label{thm1}
If $H$ is a non-cosemisimple Hopf algebra with no nontrivial skew-primitives and $|G(H)|=r$,  then
$$
\dim H \ge \min\{(2d+2)r+2lcm(d^2,r)\ |\ d>1\}.
$$
\end{theorem}
\begin{proof}
First we take the lower bound of the possible dimensions of each necessary block. Then add them together (see that these sums are related to $d$). Finally pick the minimal one for all $d>1$ as the lower bound of $\dim H$.

We have $\dim B^{1,1}_0 = |G(H)|=r$. Since $n$ divides the dimension of every block and $d^2$ divides $\dim H_{0,d}$, we have $\dim B^{d,d}_0=\dim H_{0,d}$ is a multiple of $lcm(d^2,r)$. Similarly $\dim B^{d,d}_k$ is a multiple of $lcm(d^2,r)$. By Proposition \ref{prop3} (1), $\dim B^{1,1}_m = |G(H)|=r$.

The dimension of a simple $H_0$-bicomodule in $B^{d,1}_1$ is $d$, and with left multiplications of group like elements in $G(H)$, we have another $|G(h)|-1$ simple $H_0$-bicomodules of different isomorphism types in $B^{d,1}_1$. Hence $\dim B^{d,1}_1$ is a multiple of $d|G(H)|=dr$. By Corollary \ref{cor1}, $\dim B^{1,d}_1 =\dim B^{d,1}_1$ is also a multiple of $dr$.
\end{proof}

For further discussions, we make some dotations at first.
\begin{definition}\label{def2}
For $r,d_1,d_2\ge1$,
\begin{enumerate}
\item set
$$
f(r,d_1,d_2)=\left\{
\begin{array}{ll}
r& d_1=1,d_2=1,\\
2d_2r& d_1=1,d_2\ne1,\\
2d_1r& d_1\ne1,d_2=1,\\
lcm(d_1d_2,n)& d_1\ne1,d_2\ne1;
\end{array}\right.
$$
\item set $L_{r,d_1,d_2}=\Bbbk^{\oplus f(r,d_1,d_2)}$ as a $\Bbbk$-space of dimension $f(r,d_1,d_2)$ and call it \textit{a basic block};
\item for $d>1$, set $L(r,d)=(L_{r,1,1})^{\oplus 2}\oplus (L_{r,d,d})^{\oplus 2}\oplus L_{r,d,1}$ and call it \textit{a minimal form}.
\end{enumerate}
\end{definition}

For $|G(H)|=r$ and $d,d_1,d_2>1$, $n\ge0$, we have
\begin{enumerate}
\item $\dim L(r,d) = (2d+2)r+2lcm(d^2,r)$; 
\item $\dim B^{1,1}_n$ is a multiple of $r=\dim L_{r,1,1}$;
\item $\dim B^{d,1}_n\oplus B^{1,d}_n$ is a multiple of $2dr=\dim L_{r,d,1}=\dim L_{r,1,d}$;
\item $\dim B^{d_1,d_2}_n$ is a multiple of $lcm(d_1d_2,r)=\dim L_{r,d_1,d_2}$.
\end{enumerate}
Then as $\Bbbk$-spaces, $B^{1,1}_n$, $B^{d,1}_n\oplus B^{1,d}_n$ and $B^{d_1,d_2}_n$ are isomorphic to a multiple of $L_{r,1,1}$, $L_{r,d,1}$ and $L_{r,d_1,d_2}$, respectively.

Assume that $H$ is non-cosemisimple and has no nontrivial skew-primitives. Then as $\Bbbk$-spaces, the block system of $H$ is isomorphic to a minimal form $L(|G(H)|,d)$ with more (or no) basic blocks being added. Of course, the adding of basic blocks into the minimal form must coincide with `rules' of block system.

Compare Figure 1 and Figure 2 to see the corresponding relations between necessary blocks and basic blocks.

\unitlength=1cm
$$
\begin{picture}(8.5,4.5)(0.1,0)
\put(0,0){\framebox(2.5,0.7)[]{$\mathbf{L_{r,1,1}}$}}
\put(3,0){\framebox(2.5,0.7)[]{$\mathbf{L_{r,d,d}}$}}
\put(3,1){\framebox(5.5,0.7)[]{$\mathbf{L_{r,d,1}}$}}
\put(3,3){\framebox(2.5,0.7)[]{$\mathbf{L_{r,1,1}}$}}
\put(6,3.5){\framebox(2.5,0.7)[]{$\mathbf{L_{r,d,d}}$}}
\put(4.3,2){\circle*{0.1}}
\put(4.3,2.3){\circle*{0.1}}
\put(4.3,2.6){\circle*{0.1}}
\put(7.3,2.3){\circle*{0.1}}
\put(7.3,2.6){\circle*{0.1}}
\put(7.3,2.9){\circle*{0.1}}
\put(6.8,0.4){\circle*{0.1}}
\put(7.3,0.4){\circle*{0.1}}
\put(7.8,0.4){\circle*{0.1}}
\end{picture}
$$
\begin{center}
(Figure 2)
\end{center}
\vskip5mm

With more condition on $|G(H)|$, we have the next result following Proposition \ref{prop3} and Theorem \ref{thm1}. See that $|G(H)|$ divides $\dim H$.

\begin{corollary}\label{cor5}
If $H$ is a non-cosemisimple Hopf algebra with no nontrivial skew-primitives and $|G(H)|=p$ with $p$ being a prime number, set $\dim H=tp$ with $t\ge 1$, then we have
\begin{enumerate}
\item if $p=2$, then $t\ne 1,2,\cdots,9,11,13,15$;
\item if $p=3$, then $t\ne 1,2,\cdots,13,15,16,19$;
\item If $p>3$, then $t\ne 1,2,\cdots,13,15,16,17,19\ (\text{if $p\ne5$}),20,21\ (\text{if $p\ne7$})$.
\end{enumerate}
\end{corollary}
\begin{proof}
We only prove (2) here, since the proof for (1),(2) and (3) are similar.

In (2), with $p=3$, the lower bound of $\dim H$ by Theorem \ref{thm1} is $14p$, which corresponds to two different minimal forms $L(p,2)$ and $L(p,3)$. Then $t\ne 1,2,\cdots,13$ are proved.

See that $L(p,4)=42p$. Hence, if $14p<\dim H< 42p$, then as $\Bbbk$-spaces, the block system of $H$ must be isomorphic to $L(p,2)$ or $L(p,3)$ with basic blocks added.

If $t=15,16,$ or $19$, it is not difficult to see that there are at least one basic block $L_{p,1,1}$ of $\dim p$ being added.

By Proposition \ref{prop3} (1), the block $B^{1,1}_m$ with $m= \max \{m'\ |\ B^{1,1}_{m'}\ne 0\}$ has a fixed dimension $p$. This means the new added basic block $L_{p,1,1}$ leads to a nonzero block $B^{1,1}_n$ of $H$ with $n<m$. Then for $l= \min \{m'\ |\ B^{1,1}_{m'}\ne 0\}$, $l<m$. By Proposition \ref{prop3} (2), there exists $B^{1,d}_{l'}\ne 0$, $B^{d,1}_{l'}\ne0$ and $B^{d,d}_{l'-1}\ne0$ for some $l'>l>1$ and $d>1$. Each of the three blocks is not one of the six necessary blocks of $H$. Then there are more basic blocks need to be added into $L(p,2)$ or $L(p,3)$.

For $L(p,2)$, $\dim H \ge \dim L(p,2) + \dim L(p,1,1)+ \dim L(p,2,1)+ \dim L(p,2,2)= 21p$;
for $L(p,3)$, $\dim H \ge \dim L(p,3) + \dim L(p,1,1)+ \dim L(p,3,1)+ \dim L(p,3,3)= 24p$.
Both contradict with $\dim H = tp$. So (2) is proved.
\end{proof}

Corollary \ref{cor5} provides a new angle to look at things we have already known. For example, the fact proved in \cite{DAIJIRO2011HOPF} that there is no non-cosemisimple Hopf algebra of dimension $ 30$ is a direct result from this corollary, Theorem \ref{thm1} and \cite[Corollary 2.2]{ng2005hopf}. If $H$ is a noncosemisimple Hopf algebra with $\dim H= 30$, by Theorem \ref{thm1}, $|G(H)|\ne 15,10,6$, and by Corollary \ref{cor5}, $|G(H)|\ne 5,3,2$. By \cite[Corollary 2.2]{ng2005hopf}, $H$ or $H*$ is not unimodular. Then $|G(H^*)|=30$ or $|G(H)|=30$, which contradicts with $H$ being non-coseimisimple.

\begin{remark}\label{rmk3}
Some of the results in Corollary \ref{cor5} have already been proved by different methods.

Under conditions of Corollary \ref{cor5}, let $q$ be a prime number with $p<q$. By \cite{ng2005hopf}, there is no non-cosemisimple Hopf algebra of dimension $2q$. By \cite{Andruskiewitsch1998Hopf}, a non-cosemisimple Hopf algebra of dimension $p^2$ is a Taft algebra, which is pointed. By \cite{ng2004hopf} and \cite{ng2008hopf}, Hopf algebras of dimension $pq$ with $p,q$ being odd primes and $p<q\le 4p+11$ are cosemisimple. Then all cases in (1)(2)(3) that $t$ does not equal a prime number are covered.

For $p=2$, $t\ne 4$ is proved in \cite{williams1988finite}, which shows that there is no non-cosemisimple non-pointed Hopf algebra of dimension dimension $8$; $t\ne 6$ is proved in \cite{natale2002hopf} that non-cosemisimple Hopf algebras of dimension $12$ always have nontrivial primitives; $t\ne 8$ is proved in \cite{garcia2010hopf}, in which Hopf algebras of dimension $16$ are classified completely; $t\ne 9$ is proved in \cite{hilgemann2009hopf}, which shows that there is no non-cosemisimple non-pointed Hopf algebra of dimension $2q^2$; and $t\ne15$ follows from \cite{DAIJIRO2011HOPF}, which completes the classification of Hopf algebras of dimension $30$.

For $p=3$, $t\ne 4,6$ follows from \cite{natale2002hopf} and \cite{hilgemann2009hopf}, and $t\ne 8,9,10$ is well known by \cite[Proposition 3.2]{beattie2013techniques}.
\end{remark}

Following Remark \ref{rmk1}, we list results in Corollary \ref{cor5} which have not been given by others as far as we know:
\begin{theorem}\label{thm2}
If $H$ is a non-cosemisimple Hopf algebra, then we have the following results.
\begin{enumerate}
\item For a prime number $p$,
    \begin{enumerate}
    \item if $\dim H= 12p$ with $p>3$, then $|G(H)|\ne p$;
    \item if $\dim H= 15p$ with $p>5$, then $|G(H)|\ne p$;
    \item if $\dim H= 16p$ with $p\ge 3$, then $|G(H)|\ne p$;
    \item if $\dim H= 20p$ with $p>5$, then $|G(H)|\ne p$;
    \item if $\dim H= 21p$ with $p=5$ or $p>7$, then $|G(H)|\ne p$.
    \end{enumerate}
\item Assuming that $H$ has no nontrivial skew-primitives,
    \begin{enumerate}
    \item if $\dim H= 36\ \text{or}\ 45$, then $|G(H)|\ne3$;
    \item if $\dim H=75\ \text{or}\ 100$, then  $|G(H)|\ne 5$.
    \end{enumerate}

\end{enumerate}
\end{theorem}

\begin{corollary}\label{corHopfalgofdim45}
Let $H$ be a Hopf algebra of dimension $45$. If $H$ is non-cosemisimple and non-pointed, then $G(H)=3$ and the maximal pointed Hopf subalgebra of $H$ is a Taft algebra of dimension $9$.
\end{corollary}

\vskip7mm


\vskip7mm

\bibliographystyle{plain}

\end{document}